\documentclass[11pt]{article}
\usepackage[T1]{fontenc}
\usepackage{amsmath,amssymb,amsthm,mathtools}
\usepackage[margin=1.15in]{geometry}
\usepackage{hyperref}
\hypersetup{colorlinks=true,linkcolor=blue,citecolor=blue,urlcolor=blue}
\numberwithin{equation}{section}
\newtheorem{theorem}{Theorem}[section]
\newtheorem{proposition}[theorem]{Proposition}
\newtheorem{lemma}[theorem]{Lemma}
\newtheorem{corollary}[theorem]{Corollary}
\theoremstyle{definition}
\newtheorem{definition}[theorem]{Definition}
\newtheorem{question}{Question}[section]
\theoremstyle{remark}
\newtheorem{remark}[theorem]{Remark}
\newcommand{\Sn}{S^{n-1}}
\newcommand{\surf}{S_{n-1}}
\newcommand{\Fq}{F_q^{(a)}}
\newcommand{\Vt}{\widetilde V}
\newcommand{\Haus}{\mathcal H}
\newcommand{\relint}{\operatorname{relint}}
\newcommand{\interior}{\operatorname{int}}

\newcommand{\conv}{\operatorname{conv}}
\newcommand{\ind}{\mathbf 1}

\title{Stationarity equations for fractional affine area measures:\\ellipsoid rigidity and polytopal solutions}
\author{Xingyu Zhang\thanks{School of Mathematics, Center for Applied Mathematics of Guangxi (Guangxi University), Guangxi University, Nanning, Guangxi, 530004, P. R. China.}
\and Dong Qiu\thanks{School of Mathematics, Center for Applied Mathematics of Guangxi (Guangxi University), Guangxi University, Nanning, Guangxi, 530004, P. R. China. Corresponding author. E-mail: \texttt{qiudong@gxu.edu.cn}.}}
\date{}

\begin{document}
\maketitle

\begin{abstract}
Xi and Zhao established a first variation formula for affine inequalities involving radial mean bodies, and asked two geometric questions about the resulting stationarity equation: whether the bodies already known to satisfy it are the only ones, and whether its smooth solutions are exactly the ellipsoids.
We answer both.
Every strictly convex solution of the equation with $q>-1$, $q\ne0,n$ is an ellipsoid, which answers the second question affirmatively while replacing smoothness and everywhere positive Gauss curvature by strict convexity.
Every full-dimensional facet-transitive polytope satisfies the equation, and so does the polar of every vertex-transitive polytope with respect to its volume centroid; the rhombic dodecahedron, the thirteen Catalan solids and the bipyramids over regular polygons other than the octahedron provide infinitely many affinely inequivalent solutions that are not affine images of regular polytopes or balls, so the answer to the first question is negative.
The rigidity proof compares overlap integrals along a continuous Steiner path: stationarity forces an increase of order $o(t)$, whereas a nonaffine chord midpoint function forces an increase of order $t$.
For negative parameters the required comparison rests on a minimization formula over bounded measurable auxiliary sets.
\end{abstract}



\section{Introduction}\label{sec_introduction}

For nonzero parameters $-1<p<q$, a basic extremal problem is to determine the sharp upper bound for the ratio $V(R_pK)/V(R_qK)$ over convex bodies $K$ and to characterize equality.
Here $R_qK$ is the radial mean body of $K$: its radial function is, up to a normalizing factor, the $L^q$ mean over all points $x\in K$ of the radial functions of $K$ viewed from~$x$ (see Definition~\ref{def_chord_body}).
Thus $V(R_pK)/V(R_qK)$ measures how the volume of these averaged bodies depends on the averaging exponent.
The ellipsoid maximization conjecture predicts that $V(R_pK)/V(R_qK)\leq V(R_pB)/V(R_qB)$, with equality precisely for ellipsoids, where $B$ is the Euclidean unit ball~\cite[p.~35]{XiZhao}.
This problem continues the study of ellipsoidal extremizers in classical affine isoperimetric theory.
The radial mean bodies introduced by Gardner and Zhang~\cite{GardnerZhang} connect the polar projection body and the difference body through limiting cases; their volume inequalities include the Petty projection, Zhang projection, and Rogers--Shephard inequalities.
Haddad and Ludwig~\cite{HaddadLudwig} established the ellipsoid maximization result in the range $-1<p\leq n\leq q$.
Xi and Zhao~\cite[pp.~7, 35]{XiZhao} identified the general upper bound problem, when both parameters lie on the same side of $n$, as a direction for further study.

Xi and Zhao~\cite[Theorem~1.1 and Section~6]{XiZhao} introduced a variational approach to this extremal problem.
They proved a first variation formula for the fractional integral affine surface area $\Phi_q$ under continuous Wulff perturbations and represented its derivative by a fractional affine area measure, obtaining a measure equation satisfied by extremizers.
Here $\Fq(K,\cdot)$ is the measure on the sphere obtained by differentiating $\Phi_q$ along these perturbations; it is the counterpart, for $\Phi_q$, of the surface area measure $\surf(K,\cdot)$ for the volume.
One can therefore begin by classifying stationary bodies and then compare their functional values and extremal properties.
Since every extremizer is stationary, any body that fails the equation is not extremal, while every solution is a candidate; this is what makes the classification of solutions useful.
They observed on pages~7 and~37 that such a classification may give alternative proofs of known sharp inequalities and a route to upper bounds for general parameters.
When one parameter equals the dimension $n$, the identity $V(R_nK)=V(K)$ reduces the problem to the ratio $V(R_qK)/V(K)$, whose extremizers satisfy~\cite[Corollary~6.3]{XiZhao}
\begin{equation}\label{eq_proportional_intro}
F_q^{(a)}(K,\cdot)=c\,S_{n-1}(K,\cdot)
\end{equation}
where $c>0$ and $S_{n-1}(K,\cdot)$ is the usual surface area measure.
We study the solutions of this equation for $q>-1$, $q\ne0$.
The case $q=n$ is degenerate: the ratio is then constant, $V(R_nK)/V(K)=1$, and every convex body satisfies the proportionality equation (Proposition~\ref{prop_degenerate_index}).
Proposition~\ref{prop_stationary} identifies the equation, in the nondegenerate range, with Wulff stationarity of a normalized fractional affine functional.
The basic definitions and normalization conventions are given in Section~\ref{sec_preliminaries}.

The first step is to understand how strongly the proportionality equation constrains the shape of a convex body.
Xi and Zhao~\cite[discussion after Corollary~6.3]{XiZhao} noted that ellipsoids and simplices satisfy the equation and observed the same property for regular polytopes and their $SL(n)$ images.
They then posed the following two questions.

\begin{question}[{Xi and Zhao~\cite[p.~36]{XiZhao}}]\label{q_polytopes}
Are regular polytopes, balls, and their $SL(n)$ images the only convex bodies whose surface area measure and $p$-th fractional affine area measure are proportional to each other?
\end{question}

In other words, do regular polytopes, balls, and their $SL(n)$ images exhaust the solutions of~\eqref{eq_proportional_intro}?
Here regularity is understood in the Euclidean sense; in dimension three, the regular polytopes are the five Platonic solids~\cite{Coxeter}.
The parameter $p$ in Question~\ref{q_polytopes} corresponds to our $q$.
We refer to Question~\ref{q_polytopes} as the \emph{polytope question}.

\begin{question}[{Xi and Zhao~\cite[p.~37]{XiZhao}}]\label{q_smooth}
Are ellipsoids the only convex bodies with $C^2$ boundary and everywhere positive Gauss curvature that satisfy
\[
\Vt_q(K_0,S_qK_0,z)\equiv\mathrm{const}
\]
for $z\in\partial K_0$?
\end{question}

For smooth bodies this constancy condition implies the proportionality equation, so Question~\ref{q_smooth} is the smooth case of the same problem; we call it the \emph{boundary constancy question}.

Both questions ask how large the solution set of~\eqref{eq_proportional_intro} is: whether the bodies listed in Question~\ref{q_polytopes} exhaust it, and whether the smooth constancy condition already characterizes ellipsoids.

\textbf{Main results.} We answer both questions.
The relevant objects are defined in Section~\ref{sec_preliminaries}; here and below, $\Fq(K,\cdot)$ denotes the fractional affine area measure of $K$ and $\surf(K,\cdot)$ its surface area measure.

\begin{theorem}\label{thm_strict}
Let $n\geq2$, $q>-1$, $q\ne0,n$, and let $K\subset\mathbb R^n$ be a strictly convex body.
If
\[
\Fq(K,\cdot)=c\,\surf(K,\cdot)
\]
for some constant $c>0$, then $K$ is an ellipsoid.
\end{theorem}

\begin{corollary}\label{cor_boundary_constant}
Let $K\in\mathcal K^n$ have $C^2$ boundary and everywhere positive Gauss curvature.
If $q>-1$, $q\ne0,n$ and $z\mapsto\Vt_q(K,S_qK,z)$ is constant on $\partial K$, then $K$ is an ellipsoid.
\end{corollary}

The second result shows that proportionality is not a property of regular polytopes alone.

\begin{theorem}\label{thm_facet}
Every full-dimensional facet-transitive convex polytope $P$, that is, one whose Euclidean isometry group acts transitively on its facets, satisfies
\[
\Fq(P,\cdot)=\frac{n+q}{2nV(P)}\Phi_q(P)^{n/q}\surf(P,\cdot)
\]
for all $q>-1$, $q\ne0$.
\end{theorem}

\begin{theorem}\label{thm_examples}
The rhombic dodecahedron of Proposition~\ref{prop_rhombic} and, for every $m\geq3$ with $m\ne4$, the bipyramid $Q_m$ of Proposition~\ref{prop_bipyramid} satisfy~\eqref{eq_proportional_intro} for all admissible parameters.
None of them is an affine image of a regular polytope or of a ball, and the polytopes $Q_m$ corresponding to distinct $m$ are pairwise affinely inequivalent.
Consequently the bodies listed in Question~\ref{q_polytopes} do not exhaust the solutions, and the answer to that question is negative.
\end{theorem}

Thus Corollary~\ref{cor_boundary_constant} answers Question~\ref{q_smooth} affirmatively in the nondegenerate range $q>-1$, $q\ne0,n$, and Theorem~\ref{thm_examples} answers Question~\ref{q_polytopes} negatively.
In this nondegenerate range, Theorem~\ref{thm_strict} has an extremal consequence: first-order stationarity alone determines the shape within the strictly convex class, so every strictly convex extremizer of $V(R_qK)/V(K)$ must be an ellipsoid, while any additional stationary body must fail to be strictly convex.

The polytopes of Theorem~\ref{thm_examples} satisfy the proportionality equation for every admissible parameter and are therefore stationary for the corresponding volume ratios with two parameters; see Corollary~\ref{cor_common_stationarity}.
Duality yields a much larger family of solutions: the polar of every vertex-transitive polytope with respect to its volume centroid is again a solution (Corollary~\ref{cor_duals}), so in dimension three, for instance, all thirteen Catalan solids solve~\eqref{eq_proportional_intro}.
In the range where sharp upper and lower bounds are known, these examples attain neither global extremum.
Thus first-order stationarity does not characterize extremality here: the solution set of the proportionality equation is strictly larger than the set of extremizers, and the two-parameter extremizers are not singled out by stationarity; see Remark~\ref{rem_stationarity_extremality}.
For polytopes with several facet orbits, Proposition~\ref{prop_orbits} reduces the equation to finitely many comparisons of weight ratios, which makes the criterion checkable for any given polytope.

On the polytope side, facet transitivity supplies solutions outside the proposed family and gives a sufficient symmetry condition for the proportionality equation.
On the rigidity side, Theorem~\ref{thm_strict} characterizes all strictly convex solutions in the nondegenerate parameter range, answering the boundary constancy question under a weaker regularity assumption.
Proportionality conditions on a natural functional are a classical source of rigidity: for projection functions, smooth convex bodies with proportional projection functions form a strongly constrained class~\cite{HowardHug}.
The condition studied here is of a different type, since it compares two measures carried by one and the same body, and the conclusion is that this body must be an ellipsoid.
The rigidity argument requires strict convexity; since the polytopes of Theorem~\ref{thm_examples} solve~\eqref{eq_proportional_intro} without being strictly convex, this hypothesis cannot be dropped.

The rigidity argument turns stationarity into a geometric constraint on parallel chord midpoints.
Along the continuous Steiner path, chord midpoints move while the volume remains fixed, and stationarity together with the extremal representation by auxiliary integrals implies that the increase of a weighted overlap integral is $o(t)$.
If the chord midpoint function is not affine, its midpoint defect is nonzero on a local region, which forces the same increase to be at least $c_0t$ for some $c_0>0$; this contradiction shows that the midpoint function is affine in each selected direction, and these directions form a dense set, which characterizes an ellipsoid.
The proof follows the continuous-symmetrization strategy of Milman, Shabelman, and Yehudayoff (MSY)~\cite[Section~9.1]{MSY}, together with Brock's continuous symmetrization~\cite{Brock}.
For the present functional, the Wulff variation formula must be justified along the actual continuous Steiner path.
Lemma~\ref{lem_support_expansion} establishes a uniform first-order expansion of the support function using strict convexity and transfers the variation formula to this path by a monotone sandwich argument.
The auxiliary extremal representations then connect this first-order stationarity to the overlap comparison, whose local strictness is established in Lemma~\ref{lem_local_gain}.

For negative parameters, the direct kernel integral used in the positive range diverges at the diagonal.
Lemma~\ref{lem_negative_measurable} replaces it by a finite overlap-deficit integral and expresses the fractional functional as a minimum over bounded measurable sets of fixed volume.
This admissible class contains the deformed auxiliary sets, while the Lipschitz regularity proved in Lemma~\ref{lem_negative_lipschitz} permits the use of the continuous symmetrization of MSY.
These steps make the comparison available in the negative range without assuming convexity of the auxiliary body.

Section~\ref{sec_preliminaries} recalls the basic definitions and first variation formulas.
Section~\ref{sec_energies} derives the stationarity equation and the auxiliary functionals.
Section~\ref{sec_symmetrization} contains the overlap estimates for continuous symmetrization.
Theorem~\ref{thm_strict} and Corollary~\ref{cor_boundary_constant} are proved in Section~\ref{sec_rigidity}, while Theorems~\ref{thm_facet} and~\ref{thm_examples} are proved in Section~\ref{sec_polytopes}, which also treats two-parameter stationarity.
Section~\ref{sec_conclusion} collects some open questions, and Appendix~\ref{app_technical} contains the technical proofs.

\section{Preliminaries}\label{sec_preliminaries}

Throughout, $n\geq2$. We write $V$ for $n$-dimensional volume, $\Haus^k$ for $k$-dimensional Hausdorff measure, and $du$ for spherical measure on $\Sn$.

We collect the main objects below; the references indicate where each is defined.
$K$ is a convex body with volume $V(K)$, and $\Sn$ is the unit sphere (Definition~\ref{def_support}).
The chord body $S_qK$, the radial mean body $R_qK$ and the fractional integral affine surface area $\Phi_q(K)=V(S_qK)^{q/n}$ are introduced in Definition~\ref{def_chord_body}.
The surface area measure $\surf(K,\cdot)$ is defined in Definition~\ref{def_surface_measure}, the fractional affine area measure $\Fq(K,\cdot)$ in Definition~\ref{def_fractional_measure}, and the Wulff shape $[h_K+tf]$ used for perturbations in Definition~\ref{def_wulff}.
In Section~\ref{sec_energies} we use the covariogram $G_H(z)=V(H\cap(H+z))$, the overlap integral $T_r(H,L)=\int_{rL}G_H$, the normalized auxiliary body $L_H=J(H)^{-1/n}S_qH$ of volume $1/n$, and the auxiliary functionals $E_q$ and $N_q$ whose extrema reproduce $\Phi_q$.

\begin{definition}~\cite{Schneider}\label{def_support}
A convex body is a compact convex set with nonempty interior. We denote the class of convex bodies by $\mathcal K^n$. The support function of $K$ is
\[
h_K(u)=\max_{x\in K}x\cdot u.
\]
\end{definition}

\begin{definition}~\cite{Schneider,MSY}\label{def_star_body}
For a convex body $K$, a point $z\in K$, and a direction $u\in\Sn$, define the radial function with respect to $z$ by
\[
\rho_{K,z}(u)=\max\{s\geq0:z+su\in K\}.
\]
A star body about the origin is a star-shaped set with a positive continuous radial function. For such a body $L$, write $\rho_L=\rho_{L,0}$ and define its gauge by
\[
\|x\|_L=\inf\{r>0:x\in rL\}.
\]
We call $L$ a Lipschitz star body if $\rho_L$ is Lipschitz on $\Sn$.
\end{definition}

\begin{definition}~\cite{Schneider}\label{def_surface_measure}
The outer unit normal $\nu_K$ of a convex body $K$ is defined $\Haus^{n-1}$-almost everywhere on $\partial K$. Its surface area measure is given, for a Borel set $\omega\subset\Sn$, by
\[
\surf(K,\omega)=\Haus^{n-1}(\nu_K^{-1}(\omega)).
\]
\end{definition}

\begin{theorem}~\cite{Schneider}\label{thm_minkowski_integral}
Every convex body $K$ satisfies
\begin{equation}\label{eq_minkowski_integral}
\int_{\Sn}h_K\,d\surf(K)=nV(K).
\end{equation}
\end{theorem}

We use the normalization of Xi and Zhao. Unless stated otherwise, $q>-1$ and $q\ne0$.

\begin{definition}~\cite{XiZhao,GardnerZhang}\label{def_chord_body}
For $u\in\Sn$ and $y\in u^\perp$, let $X_K(y,u)$ be the length of the chord $K\cap(y+\mathbb Ru)$, extended by zero outside the projection of $K$. The origin-symmetric chord body $S_qK$ and its associated functional are defined by
\begin{equation}\label{eq_Sq}
\rho_{S_qK}(u)^q=\int_{u^\perp}X_K(y,u)^{q+1}\,dy,
\qquad \Phi_q(K)=V(S_qK)^{q/n}.
\end{equation}
Its radial function is positive and continuous. The Gardner--Zhang radial mean body $R_qK$ is defined below, together with its relation to $S_qK$:
\begin{equation}\label{eq_SR}
S_qK=[(q+1)V(K)]^{1/q}R_qK,
\qquad
\rho_{R_qK}(u)^q=\frac1{V(K)}\int_K\rho_{K,x}(u)^q\,dx.
\end{equation}
\end{definition}

\begin{theorem}~\cite{GardnerZhang}\label{thm_positive_convexity}
For $q>0$, the radial mean body $R_qK$, and hence $S_qK$, is convex.
\end{theorem}

\begin{definition}~\cite[Definition~3.1]{XiZhao}\label{def_fractional_measure}
For an origin-symmetric star body $L$ and $z\in K$, set
\begin{equation}\label{eq_dual}
\Vt_q(K,L,z)=\frac1n\int_{\{u:\rho_{K,z}(u)>0\}}
\rho_{K,z}(u)^q\rho_L(u)^{n-q}\,du.
\end{equation}
For $q<0$, this is initially understood as an extended nonnegative integral. The fractional affine area measure is defined by
\begin{equation}\label{eq_F}
\Fq(K,\omega)=(q+1)\int_{\nu_K^{-1}(\omega)}
\Vt_q(K,S_qK,z)\,d\Haus^{n-1}(z).
\end{equation}
\end{definition}

\begin{theorem}~\cite[Definition~3.1, Corollary~3.5, and Theorem~3.11]{XiZhao}\label{thm_fractional_integral}
The measure $\Fq(K,\cdot)$ is finite and absolutely continuous with respect to $\surf(K,\cdot)$. Moreover,
\begin{equation}\label{eq_Xi_integral}
\int_{\Sn}h_K\,d\Fq(K)=\frac{n+q}{2}\Phi_q(K)^{n/q}.
\end{equation}
\end{theorem}
The bodies and measures above are invariant under translations; the reference point $z$ in~\eqref{eq_dual} is translated together with $K$. When applying Wulff variation formulas, we may first translate $K$ so that the origin lies in its interior.

\begin{definition}~\cite{Schneider}\label{def_wulff}
For $f\in C(\Sn)$, define
\[
[h_K+tf]=\bigcap_{u\in\Sn}\{x:x\cdot u\leq h_K(u)+tf(u)\}.
\]
For sufficiently small $|t|$, this set is a convex body.
We call $[h_K+tf]$ the \emph{Wulff shape} of the perturbation $h_K+tf$.
It is defined by the prescribed half-spaces above, and its support function satisfies $h_{[h_K+tf]}\leq h_K+tf$.
Equality holds when $h_K+tf$ is itself a support function.
\end{definition}

\begin{theorem}~\cite[Theorem~1.1]{XiZhao}~\cite{Schneider}\label{thm_first_variation}
For every $f\in C(\Sn)$,
\begin{align}
\left.\frac d{dt}\right|_{t=0}\Phi_q([h_K+tf])
&=2\Phi_q(K)^{(q-n)/q}\int_{\Sn}f\,d\Fq(K),\label{eq_variation_Phi}\\
\left.\frac d{dt}\right|_{t=0}V([h_K+tf])
&=\int_{\Sn}f\,d\surf(K).\label{eq_variation_volume}
\end{align}
The first identity is the Xi--Zhao variation formula, and the second is Aleksandrov's volume variation formula.
\end{theorem}

\section{Stationarity and the auxiliary functionals}\label{sec_energies}

We first identify the proportionality equation with stationarity of a scale-invariant functional. We then derive extremal representations in terms of overlap integrals, allowing stationarity to be compared with monotonicity under continuous symmetrization.

\begin{definition}\label{def_stationary}
Define the scale-invariant functional
\begin{equation}\label{eq_A}
\mathcal A_q(K)=\frac{\Phi_q(K)}{V(K)^{(n+q)/n}}.
\end{equation}
A convex body $K$ is Wulff stationary for $\mathcal A_q$ if, for every $f\in C(\Sn)$, the two-sided derivative of $\mathcal A_q([h_K+tf])$ at $t=0$ vanishes.
\end{definition}

\begin{remark}\label{rem_variation_source}
The stationarity equation for the normalization~\eqref{eq_A} follows from the Xi--Zhao first variation formula~\cite{XiZhao} and Aleksandrov's volume variation formula~\cite{Schneider}.
\end{remark}

\begin{proposition}\label{prop_degenerate_index}
For $q=n$, the functional $\mathcal A_n$ is constant, and every convex body satisfies
\begin{equation}\label{eq_qn}
F_n^{(a)}(K,\cdot)=(n+1)V(K)\surf(K,\cdot).
\end{equation}
\end{proposition}
\begin{proof}
By~\cite[Definition~2.1]{XiZhao}, $\Phi_n(K)=(n+1)V(K)^2$. Substitution into~\eqref{eq_A} gives $\mathcal A_n(K)=n+1$. Comparing~\eqref{eq_variation_Phi} and~\eqref{eq_variation_volume} then gives~\eqref{eq_qn}.
\end{proof}

\begin{proposition}\label{prop_stationary}
Let $q>-1$, $q\ne0,n$. A convex body $K$ is Wulff stationary for $\mathcal A_q$ if and only if
\begin{equation}\label{eq_stationary}
\Fq(K,\cdot)=c_q(K)\surf(K,\cdot),\qquad
c_q(K)=\frac{n+q}{2nV(K)}\Phi_q(K)^{n/q}.
\end{equation}
If the two measures are proportional with an unspecified constant, that constant necessarily equals $c_q(K)$.
\end{proposition}
\begin{proof}
Taking the logarithm of~\eqref{eq_A} and applying~\eqref{eq_variation_Phi}--\eqref{eq_variation_volume}, we obtain
\[
\left.\frac d{dt}\right|_{t=0}\log\mathcal A_q([h_K+tf])
=\int_{\Sn}f\,d\left[2\Phi_q(K)^{-n/q}\Fq(K)
-\frac{n+q}{nV(K)}\surf(K)\right].
\]
The expression in brackets is a finite signed Borel measure. Its integral against every continuous function $f$ vanishes if and only if the measure itself vanishes, which is precisely~\eqref{eq_stationary}. Alternatively, substitution of $\Fq(K)=c\surf(K)$ into~\eqref{eq_Xi_integral}, together with~\eqref{eq_minkowski_integral}, gives $cnV(K)=(n+q)\Phi_q(K)^{n/q}/2$.
\end{proof}

To normalize the volume of the auxiliary set independently of the convex body, fix $q$ and write
\begin{equation}\label{eq_normalized_auxiliary}
s_H=\rho_{S_qH},\quad J(H)=\int_{\Sn}s_H^n\,du,\quad
M(H)=J(H)^{q/n}=n^{q/n}\Phi_q(H),\quad L_H=J(H)^{-1/n}S_qH.
\end{equation}
Then $V(L_H)=1/n$ and $L_H=-L_H$. Both extremal representations below are attained at $L_H$.

The representations for positive and negative parameters both start from overlap volume. The overlap of a convex body $H$ with its translate $H+z$ is described by its covariogram $G_H$. For a bounded measurable set $L$ and $r>0$, define
\begin{equation}\label{eq_G_T}
G_H(z)=V(H\cap(H+z)),\qquad
T_r(H,L)=\int_{rL}G_H(z)\,dz
=\int_H\int_H\ind_{rL}(x-y)\,dx\,dy.
\end{equation}
Computing overlap lengths along lines parallel to $u$ gives
\begin{equation}\label{eq_covariogram_fiber}
G_H(su)=\int_{u^\perp}(X_H(y,u)-s)_+\,dy\qquad(s\geq0).
\end{equation}
\subsection{Integral representations for positive parameters}

For $q>0$, radial integration of the covariogram gives the following auxiliary functional. For an origin-symmetric convex body $L$, define
\[
E_q(H,L)=\int_H\int_H\|x-y\|_L^{q-n}\,dx\,dy.
\]
This integral is finite because the kernel is locally integrable near the diagonal and $H$ is bounded.

\begin{lemma}\label{lem_positive_energy}
For $q>0$, $q\ne n$, and an origin-symmetric convex body $L$,
\begin{equation}\label{eq_E_identity}
q(q+1)E_q(H,L)=\int_{\Sn}s_H(u)^q\rho_L(u)^{n-q}\,du.
\end{equation}
Under the normalization $V(L)=1/n$,
\begin{equation}\label{eq_positive_extreme}
M(H)=
\begin{cases}
q(q+1)\max_L E_q(H,L),&0<q<n,\\
q(q+1)\min_L E_q(H,L),&q>n.
\end{cases}
\end{equation}
The extrema are taken over origin-symmetric convex bodies $L$ of volume $1/n$, and both are attained at $L_H$. Moreover,
\begin{align}
E_q(H,L)&=(n-q)\int_0^\infty r^{q-n-1}T_r(H,L)\,dr,
&&0<q<n,\label{eq_layer_positive}\\
E_q(H,L)&=(q-n)\int_0^\infty r^{q-n-1}[V(H)^2-T_r(H,L)]\,dr,
&&q>n.\label{eq_layer_supercritical}
\end{align}
\end{lemma}
\begin{proof}[Proof sketch]
Polar coordinates, identity~\eqref{eq_covariogram_fiber} and Tonelli's theorem give~\eqref{eq_E_identity} through the elementary moment identity $\int_0^X(X-s)s^{q-1}ds=X^{q+1}/(q(q+1))$.
For the extremal representation, the volume normalization makes $d\mu=\rho_L^n\,du$ a probability measure, and Jensen's inequality applied to $Z=s_H^n/\rho_L^n$ with exponent $a=q/n$ gives the bounds in~\eqref{eq_positive_extreme}, with equality for $L=L_H$; a layer-cake decomposition of $\|x-y\|_L^{q-n}$ gives~\eqref{eq_layer_positive} and~\eqref{eq_layer_supercritical}.
The details are in Appendix~\ref{app_energy_pos}.
\end{proof}

\begin{remark}\label{rem_energy_context}
The relation between covariograms and chord integrals appears in Gardner--Zhang and Xi--Zhao~\cite{GardnerZhang,XiZhao}; related fractional energies are studied by Haddad--Ludwig~\cite{HaddadLudwig}. The identities and Jensen argument above use our volume normalization.
\end{remark}

\subsection{A minimization formula for negative parameters}

For $-1<q<0$, the kernel defining $E_q$ has a nonintegrable singularity along the diagonal. Replacing the overlap volume by its deficit gives a finite auxiliary functional on bounded measurable sets. For such a set $L$, define
\begin{equation}\label{eq_N_definition}
N_q(H,L)=\int_0^\infty r^{q-n-1}
\bigl[V(H)V(rL)-T_r(H,L)\bigr]\,dr,
\qquad C_q=-q(q+1)(n-q)>0.
\end{equation}
The bracketed expression must be integrated as a whole, since the integrals of its two terms may diverge separately.

\begin{lemma}\label{lem_negative_measurable}
For every bounded measurable set $L$, the value $N_q(H,L)$ is finite and satisfies
\begin{equation}\label{eq_N_measurable}
N_q(H,L)=-\frac1{q(q+1)}\int_L\|z\|_{S_qH}^{-q}\,dz.
\end{equation}
Among bounded measurable sets of volume $1/n$, its minimum is attained at $L_H$. In particular,
\begin{equation}\label{eq_negative_extreme}
M(H)=C_q\min_{V(L)=1/n}N_q(H,L).
\end{equation}

\end{lemma}
\begin{proof}[Proof sketch]
By~\eqref{eq_covariogram_fiber} the deficit $D_H(su)=V(H)-G_H(su)=\int_{u^\perp}\min\{s,X_H(y,u)\}\,dy$ is the measure of the part of the chord that is lost when $H$ is translated by $su$.
Interchanging the order of integration and using the elementary identity
\[
\int_0^\infty\min\{s,X\}s^{q-1}\,ds=-\frac{X^{q+1}}{q(q+1)}\qquad(X>0,\ -1<q<0)
\]
turns the defining integral~\eqref{eq_N_definition} into~\eqref{eq_N_measurable}; the resulting integrand is continuous, nonnegative and vanishes at the origin, which gives finiteness.
Since $-q>0$, the set $L_H=\lambda S_qH$ is the sublevel set of $w_H(z)=\|z\|_{S_qH}^{-q}$ at level $\lambda^{-q}$, so a sublevel-set comparison proves its minimality among sets of volume $1/n$, and multiplying the identity by $C_q$ gives~\eqref{eq_negative_extreme}.
The details are in Appendix~\ref{app_energy_neg}.
\end{proof}

For negative parameters, the minimizing set $L_H$ must also have the regularity required for continuous symmetrization. We next prove that its radial function is Lipschitz.

\begin{lemma}\label{lem_negative_lipschitz}
For every convex body $H$ and $-1<q<0$, the body $S_qH$ is an origin-symmetric Lipschitz star body.
\end{lemma}
\begin{proof}
The covariogram $G_H$ is globally Lipschitz. Indeed,
\[
|G_H(z)-G_H(w)|\leq V((H+z)\mathbin\triangle(H+w))\leq C_H|z-w|.
\]
The last inequality follows by integration along lines parallel to $z-w$: translating a section interval by $|z-w|$ produces a symmetric difference of length at most $2|z-w|$, and the projection areas of $H$ are uniformly bounded.
Choose $R$ larger than the diameter of $H$.
Then $G_H(su)=0$ for all $s\geq R$ and $u\in\Sn$.
The moment calculation in Lemma~\ref{lem_negative_measurable} gives
\[
s_H(u)^q=-q(q+1)\int_0^\infty s^{q-1}[V(H)-G_H(su)]\,ds.
\]
When comparing two directions, the tails cancel. Thus
\[
|s_H(u)^q-s_H(v)^q|
\leq -q(q+1)C_H|u-v|\int_0^R s^q\,ds.
\]
The last integral is finite because $q>-1$.
The continuous function $s_H^q$ has a positive lower bound $\eta_H$ on the compact sphere.
For $q<0$, the derivative of $a\mapsto a^{1/q}$ is bounded on $[\eta_H,\infty)$, so $s_H=(s_H^q)^{1/q}$ is Lipschitz as well.
Origin symmetry follows from $G_H(z)=G_H(-z)$.
\end{proof}

\begin{remark}\label{rem_negative_independence}
For negative parameters, we use the continuous symmetrization path for Lipschitz star bodies constructed in~\cite[Sections~4--5]{MSY}. Lemma~\ref{lem_negative_measurable} ensures that the deformed bounded measurable sets remain admissible competitors, while Lemma~\ref{lem_negative_lipschitz} supplies the regularity of the initial auxiliary body.
\end{remark}

\section{Continuous Steiner symmetrization}\label{sec_symmetrization}

We establish the two comparisons needed for rigidity.
First, strict convexity gives a uniform first-order expansion of the support function along the continuous Steiner path, allowing the stationarity formula to be applied to this path.
Second, simultaneous symmetrization of the convex body and the auxiliary set does not decrease the overlap integral.
If the chord midpoint function is not affine, the increase has a positive lower bound of order $t$.

\subsection{First variation of the support function}

Fix $e\in\Sn$ and write $Q=K\mid e^\perp$ and $D=\interior Q$. For $a\in D$, the line $a+\mathbb Re$ meets $K$ in a chord. Let $m(a)$ and $A(a)$ denote its midpoint coordinate in direction $e$ and its half-length, respectively:
\begin{equation}\label{eq_fiber}
K\cap(a+\mathbb Re)
=a+[m(a)-A(a),m(a)+A(a)]e,
\qquad A(a)>0.
\end{equation}
The lower endpoint function $f_-=m-A$ in~\eqref{eq_fiber} is convex, and the upper endpoint function $f_+=m+A$ is concave; both are continuous on $D$. Keeping chord lengths fixed and multiplying midpoint coordinates by $1-t$ defines the continuous Steiner path:
\begin{equation}\label{eq_Kt}
K_t\cap(a+\mathbb Re)
=a+[(1-t)m(a)-A(a),(1-t)m(a)+A(a)]e,
\qquad 0\leq t\leq1.
\end{equation}
The body $K_t$ is the closure of the union of these chords. Its upper and lower endpoint functions are concave and convex, respectively, since
\[
(1-t)m+A=(1-t/2)f_++(t/2)(-f_-),
\qquad
(1-t)m-A=(1-t/2)f_-+(t/2)(-f_+).
\]
Thus $K_t$ is convex, and Fubini's theorem gives $V(K_t)=V(K)$.

\begin{lemma}\label{lem_support_expansion}
If $K$ is strictly convex, the functions $m$ and $A$ extend continuously to $Q$, with $A=0$ on $\partial Q$. For each $v\in\Sn$, let $x(v)$ be the unique support point of $K$ in direction $v$, and define
\[
g(v)=-(v\cdot e)m(x(v)\mid e^\perp).
\]
Then $g\in C(\Sn)$ and
\begin{equation}\label{eq_support_o}
\|h_{K_t}-h_K-tg\|_\infty=o(t)\qquad(t\downarrow0).
\end{equation}
If $K$ also satisfies~\eqref{eq_stationary}, then
\begin{equation}\label{eq_Phi_o}
\Phi_q(K_t)-\Phi_q(K)=o(t).
\end{equation}
\end{lemma}
\begin{proof}
For $a\in\partial Q$, choose a nonzero supporting vector $\xi\in e^\perp$ such that $\xi\cdot a=\max_{b\in Q}\xi\cdot b$.
The section $K\cap(a+\mathbb Re)$ lies in the supporting hyperplane $\{x:\xi\cdot x=\xi\cdot a\}$ of $K$ and is a singleton by strict convexity.
If $a_j\in D$ and $a_j\to a$, every accumulation point of the corresponding upper and lower endpoints in the compact set $K$ belongs to this singleton.
Both endpoint sequences therefore converge to the same point.
Consequently, $f_\pm$ extend continuously to $Q$ and agree on $\partial Q$.
This proves the asserted extensions of $m,A$ and the boundary identity $A=0$.

Strict convexity also gives uniqueness of the support point $x(v)$ in every direction $v$. Compactness and uniqueness imply that $v\mapsto x(v)$ is continuous. Set
\[
\beta(v,x)=-(v\cdot e)m(x\mid e^\perp).
\]
Define the map that translates each chord by
\[
\mathcal T_t(x)=x-t\,m(x\mid e^\perp)e,\qquad x\in K.
\]
This map preserves projection points and sends the interval in~\eqref{eq_fiber} to the interval in~\eqref{eq_Kt}. Hence $K_t=\mathcal T_t(K)$ and
\[
h_{K_t}(v)=\max_{x\in K}\{v\cdot x+t\beta(v,x)\}.
\]
Choose a maximizer $x_t(v)$. We first show that $x_t(v)$ converges to $x(v)$ uniformly in $v$. Maximality gives
\[
v\cdot x_t(v)+t\beta(v,x_t(v))
\geq v\cdot x(v)+t\beta(v,x(v)).
\]
If uniform convergence failed, there would be $\eta>0$, $t_j\downarrow0$, and $v_j\in\Sn$ such that $|x_{t_j}(v_j)-x(v_j)|\geq\eta$.
Passing to a subsequence, assume $v_j\to v$ and $x_{t_j}(v_j)\to x_*$.
Since $\beta$ is bounded, passage to the limit gives $v\cdot x_*\geq v\cdot x(v)$.
Uniqueness implies $x_*=x(v)$, while $x(v_j)\to x(v)$, contradicting the distance bound.
Thus $x_t(v)-x(v)\to0$ uniformly in $v$.
Using $x(v)$ and $x_t(v)$ as competitors gives
\[
0\leq h_{K_t}(v)-h_K(v)-t\beta(v,x(v))
\leq t[\beta(v,x_t(v))-\beta(v,x(v))].
\]
The function $\beta$ is uniformly continuous on the compact set $\Sn\times K$. The right-hand side is therefore $o(t)$ uniformly in $v$, proving~\eqref{eq_support_o}.

The functional $\Phi_q$ is increasing under inclusion.
Indeed, $H\subseteq H'$ implies $X_H\leq X_{H'}$ in~\eqref{eq_Sq}.
For negative $q$, both the exponent $1/q$ in the radial function and the exponent $q/n$ in the volume functional reverse the inequality, so the final functional is still increasing.
For each $\varepsilon>0$, expansion~\eqref{eq_support_o} gives
\[
[h_K+t(g-\varepsilon)]\subseteq K_t\subseteq[h_K+t(g+\varepsilon)]
\]
for sufficiently small $t>0$. Denote the continuous linear first variation functional by
\[
\mathcal D_\Phi(f)=2\Phi_q(K)^{(q-n)/q}\int_{\Sn}f\,d\Fq(K).
\]
Applying~\eqref{eq_variation_Phi} to the two bounds and using monotonicity, we obtain
\begin{align*}
\mathcal D_\Phi(g-\varepsilon)
&\leq\liminf_{t\downarrow0}\frac{\Phi_q(K_t)-\Phi_q(K)}t\\
&\leq\limsup_{t\downarrow0}\frac{\Phi_q(K_t)-\Phi_q(K)}t
\leq\mathcal D_\Phi(g+\varepsilon).
\end{align*}
The difference between the two bounds is
\[
4\varepsilon\Phi_q(K)^{(q-n)/q}\Fq(K,\Sn),
\]
which tends to zero as $\varepsilon\downarrow0$ by the finiteness in Theorem~\ref{thm_fractional_integral}. The right derivative therefore exists and satisfies
\begin{equation}\label{eq_Phi_right}
\left.\frac d{dt}\right|_{t=0+}\Phi_q(K_t)
=2\Phi_q(K)^{(q-n)/q}\int_{\Sn}g\,d\Fq(K).
\end{equation}
Applying the same sandwich argument to volume, with $\mathcal D_V(f)=\int f\,d\surf(K)$ in place of $\mathcal D_\Phi$, and using $V(K_t)=V(K)$, gives $\int g\,d\surf(K)=0$. Equation~\eqref{eq_stationary} then makes the right derivative in~\eqref{eq_Phi_right} vanish, proving~\eqref{eq_Phi_o}.

\end{proof}

\begin{remark}\label{rem_stationary_framework}
The use of stationarity and continuous symmetrization to derive geometric constraints follows~\cite[Section~9.1]{MSY}. Here the uniform expansion~\eqref{eq_support_o} and monotone sandwich argument transfer the two-sided Wulff variation formula of Xi--Zhao~\cite[Theorem~1.1]{XiZhao} to the right derivative along the actual path. We take $\Omega=\Sn$; the perturbations $g\pm\varepsilon$ are continuous, the remainder is zero, and sufficiently small perturbations still define convex bodies.
\end{remark}

\subsection{Interval overlap estimates}

Overlap integrals can be computed line by line. Since sections of convex bodies are intervals, we first consider the relative positions of three intervals. For $\alpha,\beta,\gamma>0$, define
\begin{equation}\label{eq_Psi}
\Psi_{\alpha,\beta,\gamma}(d)
=\int_{-\alpha}^{\alpha}\int_{-\beta}^{\beta}
\ind_{[-\gamma,\gamma]}(v-w-d)\,dv\,dw.
\end{equation}

\begin{lemma}\label{lem_overlap}
Let $\alpha,\beta,\gamma>0$. The function $\Psi_{\alpha,\beta,\gamma}$ is even and nonincreasing as a function of $|d|$. If
\begin{equation}\label{eq_overlap_conditions}
|\alpha-\beta|<\gamma-|d|,
\qquad \gamma+|d|<\alpha+\beta,
\end{equation}
then
\begin{equation}\label{eq_Psi_quadratic}
\Psi_{\alpha,\beta,\gamma}(d)
=2\gamma(\alpha+\beta)-|\alpha-\beta|^2-\gamma^2-d^2.
\end{equation}
Consequently, for $0\leq t\leq1$,
\begin{equation}\label{eq_strict_overlap}
\Psi_{\alpha,\beta,\gamma}((1-t)d)
-\Psi_{\alpha,\beta,\gamma}(d)
=(2t-t^2)d^2\geq td^2.
\end{equation}
\end{lemma}
\begin{proof}
The convolution $f=\ind_{[-\alpha,\alpha]}*\ind_{[-\beta,\beta]}$ is even and satisfies
\[
f(s)=
\begin{cases}
2\min\{\alpha,\beta\},&|s|\leq|\alpha-\beta|,\\
\alpha+\beta-|s|,&|\alpha-\beta|<|s|<\alpha+\beta,\\
0,&|s|\geq\alpha+\beta.
\end{cases}
\]
By~\eqref{eq_Psi}, $\Psi(d)=\int_{d-\gamma}^{d+\gamma}f(s)\,ds$, so $\Psi$ is even.
Since $f$ is nonincreasing in $|s|$, we have $\Psi'(d)=f(d+\gamma)-f(d-\gamma)\leq0$ for $d\geq0$.
Conditions~\eqref{eq_overlap_conditions} place the two endpoints of the integration interval in the sloping parts of $f$.
As $\int_{\mathbb R}f=4\alpha\beta$, subtracting the two triangular tails gives
\[
\Psi(d)=4\alpha\beta
-\frac{(\alpha+\beta-\gamma+d)^2+(\alpha+\beta-\gamma-d)^2}{2}.
\]
Expansion gives~\eqref{eq_Psi_quadratic}. Conditions~\eqref{eq_overlap_conditions} remain valid as $|d|$ decreases. Substitution of $d$ and $(1-t)d$ and subtraction therefore give~\eqref{eq_strict_overlap}.
\end{proof}

\begin{remark}\label{rem_box_slice}
This interval integral is equivalent to the three-dimensional box section calculation in~\cite[Lemma~7.2]{MSY}. If $R=[-\alpha,\alpha]\times[-\beta,\beta]\times[-\gamma,\gamma]$, then
\[
\Haus^2(R\cap\{v-w-z=d\})=\sqrt3\,\Psi_{\alpha,\beta,\gamma}(d).
\]
We use the explicit quadratic estimate in the nondegenerate range~\eqref{eq_overlap_conditions}.
\end{remark}

\subsection{Continuous symmetrization of the auxiliary set}

To apply the interval estimate to $T_r$, we symmetrize $K$ and $L$ in the same direction simultaneously.
If $L$ is convex, each section over the interior of its projection is a single interval, so construction~\eqref{eq_Kt} applies directly and $L_t$ remains convex.
Sections of a general Lipschitz star body may consist of several intervals, which requires a rule for mergers.
Both cases are covered by the following facts, valid for origin-symmetric Lipschitz star bodies and taken from MSY~\cite[Section~4.2 and Sections~4--5]{MSY}.

\begin{enumerate}
\item[(F1)] \emph{Sections.} For a set of directions of full spherical measure there is a \emph{multi-graphical} representation: pairwise disjoint open sets $\Omega_k$ in the projection space cover the projection up to a null set, and over $\Omega_k$ every line section consists of exactly $k$ intervals with continuous endpoint functions $f_1<g_1<\cdots<f_k<g_k$~\cite[Definition~4.2 and Theorem~5.7]{MSY}; see also Lin--Xi~\cite{LinXi}.
\item[(F2)] \emph{Flow.} The continuous Steiner symmetral $L_t=S_e^tL$ is defined sectionwise. Before the first contact, the interval centres are multiplied by $1-t$. At contact, touching intervals are merged; the merged family then evolves recursively with the rescaled time parameter in~\eqref{eq_interval_flow}. The body $L_t$ is the closure of the resulting sections~\cite[Definition~4.7]{MSY}. Taking the closure does not change the sections covered by (F1)~\cite[Corollary~4.10]{MSY}; the deformation preserves volume~\cite[Corollary~4.12]{MSY} and preserves the class of Lipschitz star bodies~\cite[Proposition~5.8]{MSY}. It commutes with reflection in the origin and with positive dilations.
\item[(F3)] \emph{Convexity.} If $L$ is convex, then (F2) reduces to the chord contraction~\eqref{eq_Kt} already used for $K$, and this holds in every direction~\cite[Remark~5.10]{MSY}.
\end{enumerate}

The sectionwise description in (F2) is written out in Appendix~\ref{app_flow}, together with the elementary details used below.

\begin{lemma}\label{lem_overlap_monotonicity}
Let $K$ be a convex body, $L$ an origin-symmetric Lipschitz star body, and $e$ a direction in which $L$ is multi-graphical. Along the paths $K_t$ and $L_t$ above, for every $r>0$,
\begin{equation}\label{eq_T_monotonicity}
T_r(K_t,L_t)\geq T_r(K,L).
\end{equation}
If $L$ is convex, this holds in every direction.
\end{lemma}
\begin{proof}[Proof sketch]
Fix $r>0$ and a pair of projection points $(a,b)$ outside the exceptional null set of (F1), and put $c=(a-b)/r$.
If $c$ lies outside the projection of $L$ the contribution vanishes throughout the path; otherwise the section of $L$ over $c$ is a finite union of intervals with initial centres $\ell_i(c)$ and half-lengths $B_i(c)$. Before the first contact, its contribution to $T_r$ is
\[
\Gamma_t(a,b,r)=\sum_i\Psi_{A(a),A(b),rB_i(c)}
\bigl((1-t)[m(a)-m(b)-r\ell_i(c)]\bigr).
\]
The interval centres and the chord midpoints are contracted by the same factor, so the argument of each $\Psi$ moves monotonically towards the origin; since $\Psi$ is nonincreasing in the modulus of its argument, each term is nondecreasing in $t$.
Touching intervals are then merged; the same comparison applies after every merger, and each merger strictly reduces the number of intervals, so finitely many stages cover $[0,1]$.
Fubini's theorem over all pairs gives~\eqref{eq_slice_gain} and hence the statement.
Appendix~\ref{app_flow} contains the details, including the exceptional set and the merge times.
\end{proof}

\subsection{Affineness of chord midpoint functions}

The preceding comparison only shows that the overlap integral does not decrease. We now prove that an increase of order $o(t)$ in the weighted integral forces the chord midpoint function to be affine: any nonzero local midpoint defect gives a lower bound $c_0t$, contradicting this estimate.

\begin{lemma}\label{lem_local_gain}
Let $K$ be strictly convex, let $L=-L$ be a Lipschitz star body, and let $e$ satisfy the assumptions of Lemma~\ref{lem_overlap_monotonicity}. If
\begin{equation}\label{eq_integrated_o}
\int_0^\infty r^{q-n-1}[T_r(K_t,L_t)-T_r(K,L)]\,dr=o(t)
\end{equation}
holds for some fixed $q\ne n$, then the chord midpoint function $m$ of $K$ in direction $e$ is affine on $Q$.
\end{lemma}
\begin{proof}
By~\cite[Proposition~5.6]{MSY}, for the direction $e$ fixed above there is $\delta>0$ such that every line section of $L$ over $|c|<\delta$ is a single interval whose endpoints depend continuously on the projection point~$c$.
We write these sections as $[\ell(c)-B(c),\ell(c)+B(c)]$, where $B(c)>0$.
Origin symmetry gives $\ell(0)=0$ and $B(0)>0$.

Fix any $a_0\in D$, put $h=A(a_0)>0$, and choose $r_0=h/B(0)$ so that the three relevant intervals have equal half-lengths at this point. For projection points and radii near $(a_0,a_0,r_0)$, define the midpoint defect
\begin{equation}\label{eq_Delta}
\Delta(a,b,r)=m(a)-m(b)-r\ell((a-b)/r).
\end{equation}
At $(a_0,a_0,r_0)$, the three half-lengths $A(a),A(b),rB((a-b)/r)$ equal $h$, and the defect in~\eqref{eq_Delta} vanishes.
By continuity, there are an open ball $U\Subset D$ centered at $a_0$ and an interval $I\Subset(0,\infty)$ containing $r_0$ such that the three half-lengths lie in $(7h/8,9h/8)$ and $|\Delta|<h/4$ on $U\times U\times I$.
Shrink these neighborhoods so that $|(a-b)/r|<\delta$.
Set $\alpha=A(a)$, $\beta=A(b)$, $\gamma=rB((a-b)/r)$, and $d=\Delta(a,b,r)$.
Then
\[
|\alpha-\beta|<\frac h4<\frac{5h}8<\gamma-|d|,
\qquad
\gamma+|d|<\frac{11h}8<\frac{14h}8<\alpha+\beta.
\]
Hence~\eqref{eq_overlap_conditions} holds throughout this neighborhood.

If $\Delta$ is nonzero at some point, continuity gives a nonempty open set $O\Subset U\times U\times I$ and $d_0>0$ such that $|\Delta|\geq d_0$ on $O$.
Neither $O$ nor $d_0$ depends on $t$.
For each fixed $r$, the exceptional pairs in Lemma~\ref{lem_overlap_monotonicity} form a null set.
A second application of Fubini's theorem shows that the exceptional triples are null in $D\times D\times I$.
Removing them from $O$ gives a measurable set $O_*$ with $|O_*|=|O|>0$.

For triples in $O_*$, the entire line section of $L$ is a single interval, so~\eqref{eq_interval_flow} gives its center as $(1-t)\ell(c)$. By~\eqref{eq_strict_overlap}, $\Gamma_t-\Gamma_0\geq td_0^2$. All other contributions, outside the exceptional null set, are nonnegative by the pointwise argument that leads to~\eqref{eq_slice_gain}. Restricting the full integral first to $I$ and then to $O_*$ gives
\[
\begin{aligned}
&\int_0^\infty r^{q-n-1}[T_r(K_t,L_t)-T_r(K,L)]\,dr\\
&\quad\geq t d_0^2\int_{O_*}r^{q-n-1}\,da\,db\,dr=:c_0t.
\end{aligned}
\]
Since $I$ is bounded away from $0$ and $\infty$, the weight has a positive lower bound on $I$, so $c_0>0$. This contradicts~\eqref{eq_integrated_o}.

Thus $\Delta=0$ on $U\times U\times I$. Taking $r=r_0$ gives
\begin{equation}\label{eq_local_Cauchy}
m(a)-m(b)=r_0\ell((a-b)/r_0)\qquad(a,b\in U).
\end{equation}
This difference depends only on $a-b$.
For $x,y\in U$ and $z=(x+y)/2$, identity~\eqref{eq_local_Cauchy} gives $m(x)-m(z)=m(z)-m(y)$.
The continuous function $m$ therefore satisfies the midpoint Jensen equation on $U$ and is affine there.
Since $a_0$ was arbitrary, $m$ is locally affine on the connected open convex set $D$.
The affine coefficients agree on overlapping neighborhoods, giving a single affine function on $D$.
Boundary continuity from Lemma~\ref{lem_support_expansion} extends this function to $Q$.
\end{proof}

\section{Ellipsoid rigidity}\label{sec_rigidity}

A convex body is strictly convex if its boundary contains no nontrivial line segment. In this section we prove Theorem~\ref{thm_strict}, so $K\subset\mathbb R^n$ is strictly convex and $q>-1$, $q\ne0,n$ are fixed throughout.

By Lemma~\ref{lem_local_gain}, an increase of order $o(t)$ in the weighted overlap integral implies that the chord midpoint function is affine. We first show that this geometric property in a dense set of directions characterizes ellipsoids, and then use the extremal representations of the auxiliary functionals to complete the proof.

\begin{lemma}\label{lem_affine_midpoints}
Let $K$ be a strictly convex body. If its chord midpoint function is affine in every direction belonging to a dense subset of $\Sn$, then $K$ is an ellipsoid.
\end{lemma}
\begin{proof}
In each such direction, write $m_e(a)=b_e\cdot a+d_e$. The affine reflection
\[
\mathcal R_e(a+se)=a+[2m_e(a)-s]e
\]
preserves $K$. Its linear part has determinant of absolute value one and a one-dimensional $-1$ eigenspace parallel to $e$. Each $\mathcal R_e$ fixes the volume centroid
\[
p=\frac1{V(K)}\int_Kx\,dx.
\]
After translating the centroid $p$ to the origin, denote the linear part by $A_e$. The positive definite covariance matrix
\[
\Sigma=\frac1{V(K)}\int_K(x-p)(x-p)^{\mathsf T}\,dx
\]
satisfies $A_e\Sigma A_e^{\mathsf T}=\Sigma$. Hence $O_e=\Sigma^{-1/2}A_e\Sigma^{1/2}$ is an orthogonal reflection with normal parallel to $\Sigma^{-1/2}e$.

The body $\Sigma^{-1/2}(K-p)$ is invariant under these reflections.
An invertible linear map preserves density in the space of directions, so their normals form a dense set.
Continuity of orthogonal reflections in their normals and closedness of the body extend this invariance to reflection across every hyperplane through the origin.
Its radial function is therefore constant, and the body is a ball.
Undoing the affine transformation proves the lemma.
\end{proof}

\begin{remark}\label{rem_midpoint_geometry}
The characterization by affine chord midpoints is discussed in~\cite[Section~7.4, Theorem~7.8]{MSY} and originates with Soltan. The covariance normalization above gives a direct proof for a dense set of directions.
\end{remark}

\begin{proof}[Proof of Theorem~\ref{thm_strict}]
By Proposition~\ref{prop_stationary}, $K$ satisfies~\eqref{eq_stationary}. Lemma~\ref{lem_support_expansion} gives
\begin{equation}\label{eq_M_o}
M(K_t)-M(K)=o(t)
\end{equation}
in every direction.
Take $L=L_K$ as in~\eqref{eq_normalized_auxiliary}.
For positive parameters, $L$ is convex and any direction may be used.
For $-1<q<0$, Lemma~\ref{lem_negative_lipschitz} makes $L$ a Lipschitz star body; select any direction from the set of full measure in which it is multi-graphical.
Deform $K$ and $L$ simultaneously.
The auxiliary volume remains $1/n$.

For positive $q$, apply Lemma~\ref{lem_positive_energy}. When $0<q<n$, the maximization formula~\eqref{eq_positive_extreme}, together with~\eqref{eq_T_monotonicity} and~\eqref{eq_M_o}, gives
\begin{equation}\label{eq_squeeze_subcritical}
0\leq q(q+1)[E_q(K_t,L_t)-E_q(K,L)]
\leq M(K_t)-M(K)=o(t).
\end{equation}
For $q>n$, the minimization formula gives
\begin{equation}\label{eq_squeeze_supercritical}
0\leq q(q+1)[E_q(K,L)-E_q(K_t,L_t)]
\leq M(K)-M(K_t)=o(t).
\end{equation}
Substitute~\eqref{eq_layer_positive} into~\eqref{eq_squeeze_subcritical}, or~\eqref{eq_layer_supercritical} into~\eqref{eq_squeeze_supercritical}, and use preservation of the volume of $K_t$. Both comparisons give~\eqref{eq_integrated_o}.

For $-1<q<0$, the set $L_t$ is admissible in the measurable-set minimization~\eqref{eq_negative_extreme}. Therefore,
\begin{equation}\label{eq_squeeze_negative}
0\leq C_q[N_q(K,L)-N_q(K_t,L_t)]
\leq M(K)-M(K_t)=o(t).
\end{equation}
By Lemma~\ref{lem_negative_measurable}, both $N_q(K,L)$ and $N_q(K_t,L_t)$ are finite. Subtract the full bracketed expressions in~\eqref{eq_N_definition}. Since both volumes are preserved, we obtain
\begin{equation}\label{eq_N_difference}
N_q(K,L)-N_q(K_t,L_t)
=\int_0^\infty r^{q-n-1}[T_r(K_t,L_t)-T_r(K,L)]\,dr.
\end{equation}
First subtract on a truncated interval $(\varepsilon,R)$, then pass to the limit using finiteness of the two integrals and nonnegativity of the difference integrand. Equations~\eqref{eq_squeeze_negative}--\eqref{eq_N_difference} give~\eqref{eq_integrated_o}.

Lemma~\ref{lem_local_gain} now implies that the chord midpoint function is affine in every selected direction. A set of full spherical measure is dense, so the directions selected for negative parameters are also dense. Lemma~\ref{lem_affine_midpoints} completes the proof.
\end{proof}

Strict convexity is used in Lemma~\ref{lem_support_expansion}: it gives uniqueness of support points and singleton line sections over the projection boundary, which yield the uniform first-order expansion of the support function.

Ellipsoids satisfy the midpoint relation obtained in the proof. Center an ellipsoid $K$ at its centroid. By the $SL(n)$ equivariance and dilation rule for $S_q$~\cite[Definition~2.1 and the ensuing discussion]{XiZhao}, $L_K$ is homothetic to $K$. In a fixed direction, write $m(a)=\zeta\cdot a$. The midpoint function of the auxiliary body is then $\ell(c)=\zeta\cdot c$, so
\[
m(a)-m(b)-r\ell((a-b)/r)=0.
\]
The increase in~\eqref{eq_strict_overlap} consequently vanishes.

Finally, we apply rigidity to the boundary constancy question, Question~\ref{q_smooth}.

\begin{proof}[Proof of Corollary~\ref{cor_boundary_constant}]
Formula~\eqref{eq_F} makes the two measures proportional. Everywhere positive Gauss curvature implies strict convexity, so Theorem~\ref{thm_strict} applies. This answers Question~\ref{q_smooth} affirmatively in the nondegenerate parameter range.
\end{proof}

\begin{remark}\label{rem_rigidity_analogy}
Rigidity statements asserting that a proportionality condition on a natural functional forces the underlying body to be an ellipsoid have a classical analogue for projection functions~\cite{HowardHug}.
In the present setting the proportionality relates two measures carried by the same body, and the conclusion is that the body itself is an ellipsoid: first for smooth solutions, as asked in Question~\ref{q_smooth}, and then for all strictly convex ones (Theorem~\ref{thm_strict}).
\end{remark}

\section{Polytopal solutions and two-parameter stationarity}\label{sec_polytopes}

The proportionality equation for ellipsoids and simplices was noted by Xi--Zhao~\cite[discussion after Corollary~6.3]{XiZhao}.
We first prove that the property is preserved by invertible affine transformations, and then extend their observation for regular polytopes to facet-transitive convex polytopes; together with the explicit constructions below, this answers Question~\ref{q_polytopes} negatively.
For polytopes, both measures are atomic, and facet transitivity makes the facet areas equal and the corresponding fractional weights equal.

\subsection{Affine covariance}

\begin{proposition}\label{prop_affine}
Let $q>-1$, $q\ne0$, and suppose $\Fq(K,\cdot)=c\surf(K,\cdot)$. For $A\in GL(n)$, $b\in\mathbb R^n$, and $D=|\det A|$,
\begin{equation}\label{eq_affine_proportionality}
\Fq(AK+b,\cdot)=D^{n/q}c\surf(AK+b,\cdot).
\end{equation}
In particular, $SL(n)$ transformations preserve the proportionality constant.
\end{proposition}
\begin{proof}
For $T\in SL(n)$, set
\[
\theta_T(u)=\frac{T^{-\mathsf T}u}{|T^{-\mathsf T}u|},
\qquad w_T(u)=|T^{-\mathsf T}u|.
\]
Apply the weighted transformation identity in~\cite[Lemma~3.14]{XiZhao} to $TK$ and $T^{-1}$. For every $g\in C(\Sn)$, this gives
\begin{equation}\label{eq_affine_F}
\int_{\Sn}g(v)\,d\Fq(TK,v)
=\int_{\Sn}w_T(u)g(\theta_T(u))\,d\Fq(K,u).
\end{equation}
The surface area measure satisfies the same weighted transformation formula~\cite{Schneider}. Substituting the assumed proportionality into~\eqref{eq_affine_F} proves the assertion for $T$.

For a dilation $\lambda>0$, identity~\eqref{eq_Sq} gives $S_q(\lambda K)=\lambda^{(n+q)/q}S_qK$. Formulas~\eqref{eq_dual} and~\eqref{eq_F}, together with the scaling rule for surface area, yield
\[
\Fq(\lambda K)=\lambda^{n-1+n^2/q}\Fq(K),
\qquad \surf(\lambda K)=\lambda^{n-1}\surf(K).
\]
If $\det A>0$, write $A=D^{1/n}T$ with $T\in SL(n)$. For negative determinant, include an orthogonal reflection in the factorization. Orthogonal equivariance follows directly from the defining integrals, and translation invariance was noted above. This proves~\eqref{eq_affine_proportionality}.
\end{proof}

\begin{remark}\label{rem_affine_source}
Proposition~\ref{prop_affine} combines the weighted transformation formula of Xi--Zhao~\cite[Lemma~3.14]{XiZhao} with the classical transformation formula for surface area measure~\cite{Schneider}; the two weights coincide, which is what preserves proportionality.
\end{remark}

As an application of Proposition~\ref{prop_affine}, one can verify the proportionality property for ellipsoids directly.
For a ball $B$ centered at the origin, rotational invariance makes $\Vt_q(B,S_qB,z)$ constant on $\partial B$.
This constant is finite, since an infinite constant would contradict the finiteness of the measure in Theorem~\ref{thm_fractional_integral}.
Formula~\eqref{eq_F} gives the proportionality equation for $B$.
Every ellipsoid is an invertible affine image of a ball, so Proposition~\ref{prop_affine} gives the property for ellipsoids.

\subsection{Facet-transitive polytopes}

A polytope is \emph{facet-transitive} if its Euclidean isometry group acts transitively on its facets. Throughout this section, a facet is a face of codimension one. We prove Theorem~\ref{thm_facet}, so $P$ denotes a full-dimensional facet-transitive convex polytope.

\begin{proof}[Proof of Theorem~\ref{thm_facet}]
For a convex polytope $P$, write $\mathcal F(P)$ for its set of facets. Let $u_F$ be the outer unit normal to $F$ and $a_F=\Haus^{n-1}(F)$ its area. Since lower-dimensional faces have zero boundary area, formula~\eqref{eq_F} gives the atomic representations
\begin{equation}\label{eq_atomic_measures}
\surf(P,\cdot)=\sum_{F\in\mathcal F(P)}a_F\delta_{u_F},
\qquad
\Fq(P,\cdot)=\sum_{F\in\mathcal F(P)}b_{F,q}\delta_{u_F},
\end{equation}
where
\[
b_{F,q}=(q+1)\int_{\relint F}\Vt_q(P,S_qP,z)\,d\Haus^{n-1}(z).
\]
Theorem~\ref{thm_fractional_integral} gives finiteness of these weights; moreover, $b_{F,q}>0$.
Indeed, $\rho_{S_qP}$ is continuous and positive on the compact sphere $\Sn$, hence bounded below by some $\eta>0$.
For each $z\in\relint F$, the directions pointing into $P$ contain a nonempty open subset of a hemisphere, on which $\rho_{P,z}>0$; as $\rho_{P,z}$ is continuous, the integrand in~\eqref{eq_dual} is bounded below by a positive constant on a set of positive spherical measure, so $\Vt_q(P,S_qP,z)>0$.
Integrating over $F$, whose area $a_F$ is positive, proves the claim.

Translate the volume centroid of $P$ to the origin. Every Euclidean isometry preserving $P$ then belongs to $O(n)$. For $U\in O(n)$, changes of variables in the chord, spherical, and boundary integrals give
\[
S_q(UP)=U(S_qP),\qquad
\Vt_q(UP,S_q(UP),Uz)=\Vt_q(P,S_qP,z)
\]
with the second equality holding almost everywhere with respect to boundary area. Both measures in~\eqref{eq_atomic_measures} transform by pushforward under $U$, also when $\det U=-1$.

Facet transitivity therefore makes all $a_F$ equal and all $b_{F,q}$ equal, so the two atomic measures are proportional. Substituting this relation into~\eqref{eq_Xi_integral} and~\eqref{eq_minkowski_integral} determines the constant:
\begin{equation}\label{eq_facet_proportionality}
\Fq(P,\cdot)=\frac{n+q}{2nV(P)}\Phi_q(P)^{n/q}\surf(P,\cdot).
\end{equation}
This constant is positive.
\end{proof}

\begin{remark}\label{rem_symmetry_source}
The orthogonal equivariance and finiteness of the measures in~\cite{XiZhao} ensure equal atomic weights within each facet orbit. Facet transitivity then makes the two measures proportional.
\end{remark}

\subsection{Explicit constructions of polytopal solutions}

Regular polytopes and regular simplices are facet-transitive.
Every simplex is affinely equivalent to a regular simplex, so Theorem~\ref{thm_facet} and Proposition~\ref{prop_affine} also recover the known simplex solutions.
The following constructions lie outside the affine equivalence classes of regular polytopes.
Both are elementary: every assertion below is checked directly from the coordinates.
Invertible affine maps preserve the face lattice, the numbers of facets and vertices, and vertex degrees.
These combinatorial invariants distinguish the required examples.

\begin{proposition}\label{prop_rhombic}
Let
\[
P=\{x\in\mathbb R^3:|x_1|+|x_2|\leq1,
\ |x_1|+|x_3|\leq1,\ |x_2|+|x_3|\leq1\}.
\]
Then $P$ is a facet-transitive rhombic dodecahedron with $12$ facets and $14$ vertices. It satisfies~\eqref{eq_facet_proportionality} for every $q>-1$, $q\ne0$, but is not an affine image of any regular convex polytope or ball.
\end{proposition}
\begin{proof}
On the supporting plane $x_1+x_2=1$, write $x_1=s$ and $x_2=1-s$. Its intersection with $P$ is described by
\[
0\leq s\leq1,\qquad |x_3|\leq\min\{s,1-s\}.
\]
This is a rhombus with vertices $e_1,e_2,(1/2,1/2,1/2),(1/2,1/2,-1/2)$. Its diagonals have lengths $\sqrt2$ and $1$, so its area is $1/\sqrt2$. Coordinate permutations and independent sign changes preserve $P$ and send this facet to the other eleven. The unit normals are
\[
\left\{\frac{\varepsilon e_i+\eta e_j}{\sqrt2}:
1\leq i<j\leq3,\quad \varepsilon,\eta\in\{-1,1\}\right\}.
\]
The vertices of these facets are the six points $\pm e_i$ and the eight points $(\pm1/2,\pm1/2,\pm1/2)$.
Every vertex belongs to a facet, so there are no others.
There are $12$ quadrilateral facets, each edge belongs to two facets, and hence the number of edges is $12\cdot4/2=24$; Euler's relation reads $14-24+12=2$.
Write $d_F$ for the distance from $F$ to the origin.
Then $d_F=1/\sqrt2$, and the area calculation gives $a_F=1/\sqrt2$.
Decomposing $P$ into pyramids with bases its facets and apex the origin yields
\[
V(P)=\sum_{F\in\mathcal F(P)}\frac13a_Fd_F
=12\cdot\frac13\cdot\frac1{\sqrt2}\cdot\frac1{\sqrt2}=2.
\]
Theorem~\ref{thm_facet} therefore gives
\[
\Fq(P,\cdot)=\frac{3+q}{12}\Phi_q(P)^{3/q}\surf(P,\cdot).
\]
The facet and vertex counts of the five regular convex polytopes in dimension three are
\[
(4,4),\ (6,8),\ (8,6),\ (12,20),\ (20,12).
\]
None of them has the counts $(12,14)$, so $P$ is not affinely equivalent to any of them. An affine image of a ball is an ellipsoid, whose boundary contains no planar facet, and hence cannot be affinely equivalent to $P$ either.
\end{proof}

\begin{proposition}\label{prop_bipyramid}
For each integer $m\geq3$, define
\[
v_j=(\cos(2\pi j/m),\sin(2\pi j/m),0),\qquad
Q_m=\conv\{v_0,\ldots,v_{m-1},e_3,-e_3\}.
\]
The polytope $Q_m$ is facet-transitive and satisfies~\eqref{eq_facet_proportionality}. For $m=4$, it is the regular octahedron. For $m\ne4$, it is not an affine image of a regular convex polytope or ball. The polytopes corresponding to distinct $m$ are pairwise affinely inequivalent.
\end{proposition}
\begin{proof}
Put $c=\cos(\pi/m)>0$ and
\[
a_j=(\cos((2j+1)\pi/m),\sin((2j+1)\pi/m)).
\]
The supporting planes
\[
a_j\cdot(x_1,x_2)+cx_3=c,
\qquad a_j\cdot(x_1,x_2)-cx_3=c
\]
cut out the triangles $\conv(v_j,v_{j+1},e_3)$ and $\conv(v_j,v_{j+1},-e_3)$, respectively, with indices taken modulo $m$. All remaining generating vertices satisfy the corresponding inequality strictly. These inequalities describe the upper and lower pyramids over the regular polygon, so $Q_m$ has exactly $2m$ facets and $m+2$ vertices.

Rotation by $2\pi/m$ about the $e_3$ axis acts transitively on each family of facets, and reflection across $x_3=0$ exchanges the families. Theorem~\ref{thm_facet} applies. The $m$ equatorial edges and $2m$ edges joining the equatorial vertices to the two apices give $3m$ edges, consistent with $(m+2)-3m+2m=2$.

For $m=4$, the vertex set is $\{\pm e_1,\pm e_2,\pm e_3\}$, so $Q_4$ is the regular octahedron.
Every facet of $Q_m$ is triangular.
Thus, if $Q_m$ were affinely equivalent to a regular convex polytope, that polytope would have to be the regular tetrahedron, octahedron, or icosahedron.
Facet and vertex counts exclude the tetrahedron and leave only $m=4$ and $m=10$ for the latter two.
In particular, the counts $(6,5)$ for $m=3$ do not match any regular polytope.
For $m=10$, the two apices $\pm e_3$ have degree $10$, while the remaining ten vertices have degree $4$; every vertex of the regular icosahedron has degree $5$.
The degree counts exclude affine equivalence in this case.
An ellipsoid has no planar facet, so it cannot be affinely equivalent to $Q_m$ either.
Finally, distinct $m$ give distinct vertex counts and hence pairwise affinely inequivalent polytopes.
This completes the proof of Theorem~\ref{thm_examples}: the two constructions are solutions by Theorem~\ref{thm_facet}, and the combinatorial invariants above separate them from regular polytopes, from balls and from each other.
\end{proof}

\subsection{Duals of vertex-transitive polytopes}\label{sec_duals}

Theorem~\ref{thm_facet} converts every vertex-transitive polytope into a solution by taking its polar with respect to its volume centroid.
For a full-dimensional polytope $P$ and an interior point $z$, let
\[
P^\circ_z=\{y\in\mathbb R^n:(x-z)\cdot(y-z)\le1\ \text{for all }x\in P\}
\]
be its polar with respect to $z$. Unless another point is specified, polars in this subsection are taken with respect to the volume centroid.

\begin{lemma}\label{lem_duality}
Let $P$ be a full-dimensional polytope and let $z\in\interior P$ be fixed by every isometry preserving $P$.
Every such isometry also preserves $P^\circ_z$, and the bijection between the vertices of $P$ and the facets of $P^\circ_z$ is equivariant under these isometries.
In particular, if $P$ is vertex-transitive, then $P^\circ_z$ is facet-transitive.
\end{lemma}
\begin{proof}
Translate $z$ to the origin. Every isometry preserving $P$ then belongs to $O(n)$.
Polarity satisfies $(UP)^\circ=U(P^\circ)$ for $U\in O(n)$, since $(Ux)\cdot(Uy)=x\cdot y$.
Thus $UP=P$ implies $U(P^\circ)=P^\circ$.
Because the origin lies in the interior of $P$, its vertices are in bijection with the facets of $P^\circ$: the facet corresponding to a vertex $v$ is $F_v=\{y\in P^\circ:v\cdot y=1\}$.
For every orthogonal symmetry $U$ of $P$, the identity $U(F_v)=F_{Uv}$ proves equivariance of this correspondence.
Consequently, if the isometry group of $P$ acts transitively on its vertices, this same group acts transitively on the facets of $P^\circ$.
Translating back proves the assertions.
\end{proof}

\begin{corollary}\label{cor_duals}
The polar of every full-dimensional vertex-transitive polytope with respect to its volume centroid satisfies~\eqref{eq_proportional_intro} for all $q>-1$, $q\ne0$.
In dimension three this yields the following solutions.
\begin{enumerate}
\item[(i)] The thirteen Catalan solids, the polars of the Archimedean solids, are solutions that are not affinely equivalent to regular polytopes or balls.
\item[(ii)] The $m$-gonal trapezohedra with $m\ge4$, the polars of the uniform antiprisms, are solutions that are not affinely equivalent to regular polytopes or balls.
\item[(iii)] The $m$-gonal bipyramids, the polars of the uniform prisms, recover Proposition~\ref{prop_bipyramid}.
\end{enumerate}
\end{corollary}
\begin{proof}
The volume centroid of a full-dimensional polytope lies in its interior and is fixed by every isometry preserving it. Lemma~\ref{lem_duality} therefore makes the polar of a vertex-transitive polytope facet-transitive, and Theorem~\ref{thm_facet} gives the first assertion. The polyhedral groups of the Archimedean solids, the uniform antiprisms and prisms act transitively on their vertices~\cite{Coxeter}, so the assertion applies to all the examples listed above.

For (i), affine maps preserve the numbers of facets and of vertices, so it suffices to compare these invariants with those of the five regular polytopes.
The polar of an Archimedean solid has as many facets as the solid has vertices and as many vertices as the solid has facets.
For the truncated tetrahedron, the cuboctahedron, the truncated cube, the truncated octahedron, the rhombicuboctahedron, the snub cube, the truncated cuboctahedron, the icosidodecahedron, the truncated dodecahedron, the truncated icosahedron, the rhombicosidodecahedron, the snub dodecahedron and the truncated icosidodecahedron, the resulting pairs are
\[
\begin{aligned}
&(12,8),\ (12,14),\ (24,14),\ (24,14),\ (24,26),\ (24,38),\ (48,26),\\
&(30,32),\ (60,32),\ (60,32),\ (60,62),\ (60,92),\ (120,62),
\end{aligned}
\]
and none of them occurs among $(4,4)$, $(6,8)$, $(8,6)$, $(12,20)$, $(20,12)$; an ellipsoid, on the other hand, has no planar facet.

For (ii), the polar of the $m$-gonal antiprism has $2m$ facets and $2m+2$ vertices, and for $m\ge4$ such a pair is again not among the five; the remaining case $m=3$ gives the cube.
For (iii), the polar of the $m$-gonal prism is affinely equivalent to the bipyramid $Q_m$ of Proposition~\ref{prop_bipyramid}, whose properties have already been verified.
\end{proof}

The argument uses only vertex transitivity, so it applies in every dimension: the polar of any vertex-transitive polytope with respect to its volume centroid, for instance of a uniform one in higher dimensions, is again a solution.
Vertex transitivity is not a redundant hypothesis here: the elongated square gyrobicupola has congruent vertex figures but is not vertex-transitive, and its polar is therefore not covered by Lemma~\ref{lem_duality}.

\subsection{A facet orbit criterion}

For a polytope with several facet orbits, the proportionality equation reduces to equality of the measure weight ratios across the orbits.

\begin{proposition}\label{prop_orbits}
Center a convex polytope $P$ at its volume centroid, and let $\mathcal O_1,\ldots,\mathcal O_k$ be the facet orbits of its orthogonal symmetry group. Let $a_i>0$ and $b_{i,q}>0$ be the area and fractional measure weight of any facet in the $i$th orbit. Then
\begin{equation}\label{eq_orbit_criterion}
\Fq(P,\cdot)\propto\surf(P,\cdot)
\quad\Longleftrightarrow\quad
\frac{b_{1,q}}{a_1}=\cdots=\frac{b_{k,q}}{a_k}.
\end{equation}
\end{proposition}
\begin{proof}
Orthogonal equivariance makes facet areas and fractional weights constant within each orbit. Their positivity was proved after~\eqref{eq_atomic_measures}. Distinct facets have distinct outer unit normals, so comparison of atomic weights gives~\eqref{eq_orbit_criterion}. The common ratio is necessarily positive.
\end{proof}

Facet transitivity is sufficient for this property, but not necessary. For example, the tetrahedron
\[
P_0=\conv\{0,e_1,2e_2,3e_3\}
\]
is affinely equivalent to a regular tetrahedron and thus satisfies the proportionality equation by Theorem~\ref{thm_facet} and Proposition~\ref{prop_affine}. Its six edge lengths are $1,2,3,\sqrt5,\sqrt{10},\sqrt{13}$, all distinct. Every Euclidean symmetry must preserve each edge and hence fix each vertex. Its symmetry group is therefore trivial, and its four facets belong to distinct orbits.

This example shows that criterion~\eqref{eq_orbit_criterion} tests measure weight ratios without requiring facet transitivity. The same orbit-by-orbit comparison can be used for prisms, antiprisms, and other polytopes with several facet orbits.

\subsection{Stationarity for volume ratios with two parameters}

A convex body satisfying the proportionality equation for two parameters is also stationary for the corresponding volume ratio. For distinct nonzero parameters $p,q>-1$, define the logarithmic functional
\begin{equation}\label{eq_Jpq}
J_{p,q}(K)=-\left(\frac1p-\frac1q\right)\log V(K)
+\frac1p\log\Phi_p(K)-\frac1q\log\Phi_q(K).
\end{equation}
Identity~\eqref{eq_SR} gives the exact relation
\begin{equation}\label{eq_Jpq_ratio}
J_{p,q}(K)=\frac1n\log\frac{V(R_pK)}{V(R_qK)}
+\frac{\log(p+1)}p-\frac{\log(q+1)}q.
\end{equation}
Thus $J_{p,q}$ and the corresponding volume ratio have the same Wulff stationary bodies. Xi--Zhao~\cite[Theorem~6.2]{XiZhao} derived the measure equation for this logarithmic functional from extremality. We substitute the two single-parameter proportionality relations into that equation to obtain stationarity.

\begin{corollary}\label{cor_common_stationarity}
Let $p,q>-1$ be distinct nonzero parameters. If
\[
F_s^{(a)}(K,\cdot)=\frac{n+s}{2nV(K)}\Phi_s(K)^{n/s}\surf(K,\cdot),
\]
holds for both $s\in\{p,q\}$, then $K$ is Wulff stationary for $J_{p,q}$. In particular, every facet-transitive polytope has this stationarity property for two parameters; one of them may equal $n$.
\end{corollary}
\begin{proof}
By~\eqref{eq_variation_Phi}--\eqref{eq_variation_volume}, the first derivative of~\eqref{eq_Jpq} in the Wulff direction $f$ is
\[
\int_{\Sn}f\,d\left[
-\frac{1/p-1/q}{V(K)}\surf(K)
+\frac{2}{p\Phi_p(K)^{n/p}}F_p^{(a)}(K)
-\frac{2}{q\Phi_q(K)^{n/q}}\Fq(K)\right].
\]
When the two proportionality equations hold, the coefficient of $\surf(K)$ is
\[
-\frac{1/p-1/q}{V(K)}
+\frac{n+p}{npV(K)}-\frac{n+q}{nqV(K)}=0.
\]
This proves stationarity. If one parameter equals $n$, Proposition~\ref{prop_degenerate_index} supplies the required proportionality relation and the same cancellation applies. The assertion for facet-transitive polytopes follows from Theorem~\ref{thm_facet}. By~\eqref{eq_Jpq_ratio}, the conclusions also hold for the radial mean body volume ratio.
\end{proof}

\begin{remark}\label{rem_stationarity_extremality}
Stationarity does not characterize extremality for this functional.
For nonzero parameters $-1<p\leq n\leq q$ with $p<q$, equality in the upper bound of Haddad--Ludwig holds precisely for ellipsoids, and equality in the lower bound of Gardner--Zhang holds precisely for simplices; see also Xi--Zhao~\cite[Theorem~1.4 and the ensuing discussion]{XiZhao}.
The rhombic dodecahedron of Proposition~\ref{prop_rhombic} and the bipyramids of Proposition~\ref{prop_bipyramid} with $m\ne4$ are neither ellipsoids nor simplices, so by Corollary~\ref{cor_common_stationarity} they are stationary for the corresponding volume ratios while attaining neither the global maximum nor the global minimum.
\end{remark}

\section{Concluding remarks and open questions}\label{sec_conclusion}

We have classified the strictly convex solutions of the nondegenerate proportionality equation: they are exactly the ellipsoids.
The facet-transitive polytope theorem extends the observation of Xi and Zhao for regular polytopes to a larger family of solutions.
The rhombic dodecahedron and the bipyramids over regular polygons show that regular polytopes, balls, and their affine images do not exhaust the solutions; by Corollary~\ref{cor_duals} this remains true for the polars of all vertex-transitive polytopes, among them the thirteen Catalan solids.
For negative parameters, the rigidity comparison is supplied by minimization over bounded measurable sets and continuous symmetrization of Lipschitz star bodies.
Our results leave two things open: the classification of solutions that are not strictly convex, a class which is nonempty by the examples of Section~\ref{sec_polytopes}, and the dependence of the solution set on the parameter $q$.

For fixed $q>-1$, $q\ne0,n$, two further questions arise in the classification of solutions that are not strictly convex: (i) is every three-dimensional polytopal solution affinely equivalent to a facet-transitive polytope, and (ii) does there exist a solution with smooth boundary containing a nontrivial line segment?
Question~(i) can be studied through~\eqref{eq_orbit_criterion} for polytopes with two or more facet orbits, while question~(ii) concerns the role of strict convexity in the classification.
The number of solutions of the orbit weight equations and their dependence on $q$ also remain to be investigated.

The local extremal type of bodies stationary for volume ratios with two parameters is a further question.
Even in the parameter range where sharp upper and lower bounds are known, failure of our examples to attain a global extremum does not determine whether they are local extremizers or saddle points.
Their second variations under volume-preserving perturbations may help distinguish stationarity from extremality.

\section*{Acknowledgements}
{\sloppy This work was supported by the National Natural Science Foundation of China (Grant No.~12571489) and Guangxi Natural Science Foundation (No.~2025GXNSFAA069576).\par}

\appendix

\section{Technical proofs}\label{app_technical}

\subsection{Proof of Lemma~\ref{lem_positive_energy}}\label{app_energy_pos}

\begin{proof}
Polar coordinates, identity~\eqref{eq_covariogram_fiber}, and Tonelli's theorem give~\eqref{eq_E_identity}, since
\[
\int_0^X(X-s)s^{q-1}\,ds=\frac{X^{q+1}}{q(q+1)}.
\]
For the extremal representation, set $d\mu=\rho_L^n\,du$, $a=q/n$, and $Z=s_H^n/\rho_L^n$.
The volume normalization makes $\mu$ a probability measure.
Then $\int Z\,d\mu=J(H)$, whereas the right-hand side of~\eqref{eq_E_identity} is $\int Z^a\,d\mu$.
Jensen's inequality gives an upper bound when $0<a<1$ and a lower bound when $a>1$.
Equality holds for $L=L_H$, which belongs to the admissible class by Theorem~\ref{thm_positive_convexity}.

Finally, the layer-cake representation of $\|x-y\|_L^{q-n}$ gives~\eqref{eq_layer_positive} and~\eqref{eq_layer_supercritical}.
\end{proof}

\subsection{Proof of Lemma~\ref{lem_negative_measurable}}\label{app_energy_neg}

\begin{proof}
Identity~\eqref{eq_covariogram_fiber} gives
\[
D_H(su):=V(H)-G_H(su)=\int_{u^\perp}\min\{s,X_H(y,u)\}\,dy.
\]
Changing variables by $z=rw$ in~\eqref{eq_G_T} and applying Tonelli's theorem to the nonnegative deficit, we obtain
\[
N_q(H,L)=\int_L\int_0^\infty r^{q-1}D_H(rw)\,dr\,dw.
\]
For $w=|w|u\ne0$ the inner integral evaluates as
\[
|w|^{-q}\int_0^\infty s^{q-1}D_H(su)\,ds
=-\frac{|w|^{-q}s_H(u)^q}{q(q+1)},
\]
because $\int_0^\infty\min\{s,X\}s^{q-1}\,ds=-X^{q+1}/[q(q+1)]$ for $X>0$ and $-1<q<0$.
This proves~\eqref{eq_N_measurable}. The integrand in~\eqref{eq_N_measurable} is continuous, nonnegative and homogeneous of degree $-q$, and it vanishes at the origin; hence it is bounded on bounded sets, which proves the finiteness of $N_q(H,L)$.

Write $w_H(z)=\|z\|_{S_qH}^{-q}$ and $\lambda=J(H)^{-1/n}$. Since $-q>0$, the set $L_H=\lambda S_qH$ is exactly the sublevel set $\{w_H\leq\lambda^{-q}\}$. For any competitor $L$ of the same volume,
\[
\int_Lw_H-\int_{L_H}w_H
=\int_{L\setminus L_H}(w_H-\lambda^{-q})
-\int_{L_H\setminus L}(w_H-\lambda^{-q})\geq0.
\]
This sublevel set comparison proves the minimality of $L_H$. Polar integration also gives
\[
\int_{L_H}w_H(z)\,dz
=\frac1{n-q}\int_{\Sn}s_H^q(\lambda s_H)^{n-q}\,du
=\frac{M(H)}{n-q}.
\]
Multiplying~\eqref{eq_N_measurable} by $C_q$ yields~\eqref{eq_negative_extreme}.
\end{proof}

\subsection{The sectionwise flow and the proof of Lemma~\ref{lem_overlap_monotonicity}}\label{app_flow}

For a finite union of disjoint closed intervals $J=\bigcup_{i=1}^k[\ell_i-B_i,\ell_i+B_i]$, with $B_i>0$, let $\tau$ be the first contact time when the centers move as $(1-t)\ell_i$; for a single interval, set $\tau=1$. The time parametrization in~\cite[Section~4.2]{MSY} is
\begin{equation}\label{eq_interval_flow}
S^tJ=
\begin{cases}
\displaystyle\bigcup_{i=1}^k[(1-t)\ell_i-B_i,(1-t)\ell_i+B_i],
&0\leq t\leq\tau,\\[2pt]
\displaystyle S^{(t-\tau)/(1-\tau)}(S^\tau J),
&\tau<t\leq1\quad(\tau<1).
\end{cases}
\end{equation}
At time $\tau$, merge touching intervals and apply the formula recursively to the resulting smaller collection. Each merger reduces the number of intervals, so the recursion terminates after finitely many steps, and the total length is preserved.

Applying~\eqref{eq_interval_flow} to each section over $c\in\Omega_\infty$ and taking the closure of their union gives $L_t=S_e^tL$~\cite[Definition~4.7]{MSY}; by~\cite[Corollary~4.10]{MSY} the closure does not alter these sections, by~\cite[Corollary~4.12]{MSY} we have $V(L_t)=V(L)$, and by~\cite[Proposition~5.8]{MSY} each $L_t$ is again a Lipschitz star body.
Reflection in the origin commutes with translation, merger and closure, so $L=-L$ implies $L_t=-L_t$; positive dilations commute with the path as well.
For a convex body, each section over the interior of its projection is a single interval, its centre in~\eqref{eq_interval_flow} is $(1-t)m(a)$ and its half-length is $A(a)$, so taking the closure gives $S_e^tK=K_t$, exactly as in~\eqref{eq_Kt}~\cite[Remark~5.10]{MSY}.

\begin{proof}[Proof of Lemma~\ref{lem_overlap_monotonicity}]
Fix $r>0$, and let $Z\subset e^\perp$ be the exceptional null set in the multi-graphical representation. For fixed $b$, the map $a\mapsto(a-b)/r$ is an invertible affine map, so $\{a:(a-b)/r\in Z\}$ is null. Fubini's theorem gives
\[
\bigl|\{(a,b)\in D\times D:(a-b)/r\in Z\}\bigr|=0.
\]
For the remaining pairs, set $c=(a-b)/r$. If $c$ lies outside the projection of $L$, its contribution is zero throughout the path. Otherwise, the section is a finite union of intervals. Before the first contact, write $\ell_i(c)$ and $B_i(c)$ for their initial centers and half-lengths. Integration along the two chords of $K$ and evenness of $\Psi$ give the contribution to $T_r$ denoted by $\Gamma_t(a,b,r)$ in the sketch above; the unconditional monotonicity in Lemma~\ref{lem_overlap} shows that each term, and hence $\Gamma_t(a,b,r)$ itself, is nondecreasing in $t$.
At contact, the union of touching intervals differs from the merged interval only by repeated endpoints, so the overlap integral is continuous.
Suppose a merger occurs at $\tau<1$, and let a resulting interval have center $\ell(\tau)$ and half-length $B$ at that time.
Until the next contact, formula~\eqref{eq_interval_flow} gives
\[
\ell(t)=\frac{1-t}{1-\tau}\ell(\tau),\qquad
d(t)=(1-t)[m(a)-m(b)]-r\ell(t)
=\frac{1-t}{1-\tau}d(\tau).
\]
Thus $|d(t)|\leq|d(\tau)|$, while $A(a),A(b),rB$ remain fixed. Monotonicity of $\Psi$ in $|d|$ gives the comparison on this stage. Each merger strictly reduces the number of intervals, so finitely many iterations cover $[0,1]$ for the fixed section.

The preceding argument therefore gives $\Gamma_t(a,b,r)\geq\Gamma_0(a,b,r)$ for every remaining pair. Identity~\eqref{eq_G_T} and Fubini's theorem then yield
\begin{equation}\label{eq_slice_gain}
T_r(K_t,L_t)-T_r(K,L)
=\int_{D\times D}[\Gamma_t(a,b,r)-\Gamma_0(a,b,r)]\,da\,db\geq0.
\end{equation}
The sections covered by the multi-graphical representation are unchanged by closure, so this is the overlap integral for $L_t$. When $L$ is convex, every interior section is a single interval and the same argument applies in every direction.
\end{proof}


\begin{thebibliography}{99}
\bibitem{XiZhao}
D. Xi and Y. Zhao,
\emph{Fractional affine area measures},
J. Eur. Math. Soc., to appear (2026).

\bibitem{GardnerZhang}
R. J. Gardner and G. Zhang,
\emph{Affine inequalities and radial mean bodies},
Amer. J. Math. \textbf{120} (1998), no.~3, 505--528.
\href{https://doi.org/10.1353/ajm.1998.0021}{doi:10.1353/ajm.1998.0021}.

\bibitem{HaddadLudwig}
J. Haddad and M. Ludwig,
\emph{Affine fractional Sobolev and isoperimetric inequalities},
J. Differential Geom. \textbf{129} (2025), no.~3, 695--724.
\href{https://doi.org/10.4310/jdg/1740137450}{doi:10.4310/jdg/1740137450}.

\bibitem{MSY}
E. Milman, S. Shabelman, and A. Yehudayoff,
\emph{Fixed and periodic points of the intersection body operator},
Invent. Math. \textbf{241} (2025), 509--558.
\href{https://doi.org/10.1007/s00222-025-01342-z}{doi:10.1007/s00222-025-01342-z}.

\bibitem{Brock}
F. Brock,
\emph{Continuous Steiner-symmetrization},
Math. Nachr. \textbf{172} (1995), 25--48.
\href{https://doi.org/10.1002/mana.19951720104}{doi:10.1002/mana.19951720104}.

\bibitem{Schneider}
R. Schneider,
\emph{Convex Bodies: The Brunn--Minkowski Theory},
second expanded ed., Encyclopedia of Mathematics and its Applications, vol.~151,
Cambridge University Press, Cambridge, 2014.

\bibitem{LinXi}
Y. Lin and D. Xi,
\emph{Orlicz affine isoperimetric inequalities for star bodies},
Adv. Appl. Math. \textbf{134} (2022), Article 102308.
\href{https://doi.org/10.1016/j.aam.2021.102308}{doi:10.1016/j.aam.2021.102308}.

\bibitem{HowardHug}
R. Howard and D. Hug,
\emph{Smooth convex bodies with proportional projection functions},
Israel J. Math. \textbf{159} (2007), 317--341.
\href{https://doi.org/10.1007/s11856-007-0049-z}{doi:10.1007/s11856-007-0049-z}.

\bibitem{Coxeter}
H. S. M. Coxeter,
\emph{Regular Polytopes}, third ed., Dover Publications, New York, 1973.
\end{thebibliography}
\end{document}